\documentclass[letterpaper, 10 pt, conference]{ieeeconf}
\usepackage[utf8]{inputenc}
\usepackage[T1]{fontenc}
\usepackage[english]{babel}
\usepackage{amsmath,amsfonts,amssymb}
\usepackage{geometry} 
\usepackage{empheq} 
\usepackage{stmaryrd} 
\usepackage{graphicx} 
\usepackage{caption} 
\usepackage{color} 
\usepackage{hyperref}
\usepackage[sorting = none, backend=biber, bibstyle=ieee, citestyle=numeric-comp, dashed=false]{biblatex}
\usepackage{microtype}
\usepackage{breqn} 
\usepackage{tikz}
\usepackage{wrapfig}
\usepackage{caption}
\usepackage[ruled,vlined]{algorithm2e}
\usepackage{stmaryrd}
\usepackage{graphicx}               
\usepackage{version} 

\usepackage{accents}
\DeclareMathSizes{10}{9}{7}{7}
\AtNextBibliography{\small}

\def\RR{{\mathbb R}}    

\def\He{\normalfont{\texttt{He}}}
\IEEEoverridecommandlockouts     
\graphicspath{ {C:\Users\Mathieu\Documents\Thèse\Schema Papier\Papier 1} }

\DeclareUnicodeCharacter{0301}{\'{e}}

\newtheorem{theorem}{Theorem}
\newtheorem{hyp}{Assumption}

\newtheorem{remark}{Remark}

\title{Distributed Event-Triggered Leader-Follower Consensus of Nonlinear Multi-Agent Systems}

\author{Mathieu Marchand$^{1}$, Vincent Andrieu$^{2}$, Sylvain Bertrand$^{1}$ and H\'el\`ene Piet-Lahanier$^{1}$
	\thanks{*This work was funded by French grant ANR Delicio (ANR-19-CE23-0006).}%
    \thanks{$^{1}$M. Marchand, S. Bertrand and H. Piet-Lahanier are with Universit\' e Paris-Saclay, ONERA, Traitement de l\textquoteright information et syst\`emes, 91123, Palaiseau, France (e-mail: \{\texttt{mathieu.marchand}; \texttt{sylvain.bertrand}; \texttt{helene.piet-lahanier}\}\texttt{@onera.fr}).}%
    \thanks{$^{2}$V. Andrieu is with Universit́e Lyon,  CNRS, LAGEPP, Villeurbanne, France (e-mail: \texttt{vincent.andrieu@gmail.com}).}
}

\begin{document}

\maketitle
\begin{abstract}
We consider  the distributed leader-follower consensus problem with event-triggered communications. The system under consideration is a non-linear input-affine multi agent system. The agents are assumed to have identical dynamics structure with uncertain parameters and satisfying an incremental stabilisability condition. 
A distributed control law is proposed which achieves consensus based on two novel Communication Triggering Conditions (CTCs): the first one to achieve an asymptotic consensus but without any guarantees on Zeno behaviour and the second one to exclude Zeno behaviour but  with practical consensus. 
\end{abstract}

\section{Introduction}
As missions to be performed by autonomous systems become increasingly challenging, there is a growing interest in the use of Multi-Agent Systems (MAS) for many applications such as inspection/mapping/exploration with unmanned aerial vehicles or mobile robots, see e. g. \cite{olfati2007consensus}, \cite{verma2021multi}. The use of such systems allows to achieve complex or highly demanding tasks with agents with lower abilities through cooperation. 
MAS control may be either centralized or distributed. In the centralized case, control laws are designed on the basis of a global knowledge of the MAS which makes this approach very efficient. However, it is at the cost of being very sensitive to potential deficiency or loss of the central agent. 
In distributed control, each agent computes its control input according to local information (i.e.  from the agents located in its neighbourhood). 
At the cost of potential discrepancy of resulting performances, this strategy provides enhanced robustness to the loss of an agent.
Thus, distributed approaches have been increasingly employed, especially for cooperative MAS. This cooperation requires the definition of a communication strategy. 

Development of distributed control to make a MAS reach a consensus, i.e. a situation where the state vectors of the agents converge to an identical value, has received significant interest in the last decades. This problem can be split in two categories: leaderless consensus \cite{olfati2007consensus, zheng2017consensus} and leader-follower consensus \cite{zhang2015leader,cao2015leader,wu2018leader, giaccagli2021sufficient}. In these works, consensus schemes are based on either continuous or periodic exchanges of state values by agents within the same neighbourhood.

Continuous communications are intractable in practice and periodic ones are therefore preferred. Nevertheless they can be still energy consuming, and bandwidth problems may arise when dealing with networks composed of large number of agents.



In event-triggered communication \cite{tabuada2007event}, exchange of information between connected agents is decided by evaluating a Communication Triggering Condition (CTC) usually defined from a function of the current states and some information error.
%
Solutions to the consensus problem with event-triggered communication have been proposed for single and double integrator MAS, see e.g. \cite{seyboth2011control, seyboth2013event}. Then, the problem has also been studied for agents with general linear dynamics, as in \cite{garcia2014decentralized}. 
Several extensions have been suggeted to improve the previous result, e.g. as in \cite{viel2017new}, resulting in reduced communications.
 
Regarding nonlinear systems, 
various articles proposed distributed control laws and CTCs for MAS with Euler-Lagrange nonlinear dynamics, see e.g. \cite{huang2016distributed, liu2017event, viel2019}.
%
Distributed event-triggered control has also been investigated for agents with more general nonlinear dynamics. 
Solutions for reaching consensus with event-triggered control in MAS with nonlinear fully-actuated input-affine dynamics have been proposed in \cite{xie2015event,liu2017fixed,liu2018fixed}, or e.g. in \cite{liu2018fixed} regarding the finite-time event-triggered consensus problem.

For possibly fully actuated systems, an interesting approach has been presented in \cite{wang2019distributed} to solve an event-triggered MAS consensus problem with nonlinear input-affine dynamics. The problem is split into two simpler problems. First, the consensus problem of a virtual MAS with linear dynamics with the same graph topology as the real one is solved, considering an event-triggered mechanism. Then, the state of the virtual system can be used as reference for the tracking problem of agents of the real MAS with nonlinear dynamics. Nevertheless, the approach is restricted to  agents described by minimum phase dynamics.



In this paper, we consider the leader-follower consensus problem for nonlinear input-affine MAS. Contrary to aforementionned works, the dynamics considered here are neither assumed to be fully-actuated nor assumed to be described by minimum phase system.
Moreover, uncertainty on some parameters is considered in the system description. Two novel distributed event-triggered control schemes with associated CTCs are proposed.

The paper is organised as follows. In Section \ref{problem_statement}, the problem is defined. Main results are provided in Section \ref{main_results}. In Section \ref{result} the efficiency of the algorithms is analysed through an illustrative simulation example.

\noindent{\bf Notations.} Given a matrix $M\in\RR^{n\times m}$, we denote by $m_{ij}$ the scalar element in row $i$ and column $j$ of $M$.
The identity matrix is denoted by $I$. 
For any matrix $A$, $A^\top$ denotes its transpose matrix and, for square matrices, $\He\{A\} := A + A^\top$. 
A matrix $P$ satisfies $P>0 \;(\geq 0)$ if $P$ is symmetric and positive definite (semi-positive definite). 
The Kronecker product is denoted by $\otimes$. The largest and the smallest eigenvalue of $M$ are respectively denoted by $\lambda_{\max}(M)$ and $\lambda_{\min}(M)$. 
%
The sets of positive real numbers and non null positive real numbers are respectively denoted by $\RR_+$ and $\RR_+^*$
We also denote by $C^1$, the set of differentiable functions with continuous derivative.

\section{Problem Statement}
\label{problem_statement}
\subsection{System description}
In this paper, we consider the problem of synchronising a MAS composed by a leader agent with state $x_1$ in $\RR^n$ and $N-1$ identical followers with states $x_2,\dots,x_N$ also in $\RR^n$.
The leader's dynamics are described by
\begin{equation}\label{eq_DynLeader}
    \dot x_1(t) = f(x_1(t), \,\theta), 
\end{equation}
with $\theta \in \Omega \subset \RR^p$ a vector of  constant parameters, possibly unknown, and $f:\RR^n\times\Omega\mapsto\RR^n$ a $C^1$ vector field.
The followers' dynamics are described by
\begin{equation}
\label{dyna}
    \dot{x}_i(t) = f(x_i(t),\,\theta) + B\, u_i(t), \quad  \, i=2,\dots,N,
\end{equation}
where $x_i \in \RR^n$ is the state of the agent $i$, $u_i:t\mapsto \RR^m$ is the control input and $B$ is in $\RR^{n\times m}$.

The communication graph $\mathcal G$ between the agents is described by $\left\{ \mathcal V, \, \mathcal E, A\right\}$, where $\mathcal V = \{1,\, \dots\,, N\}$ denotes the set of vertices,  $\mathcal E\subset \{1,\dots,N\}^2$ denotes the set of edges  which represents the communication capability among agents and $A=\left(a_{ij}\right)_{(i,j)\in\{1,\dots,N\}^2}$ is the adjacency matrix which entries $a_{ij}\geq 0$ represent the weights of the communication capabilities.
We have $a_{ij}=0$ if $(i,j)\notin\mathcal E$ and $a_{ij}>0$ if $(i,j)\in\mathcal E$. If $(i,j) \in \mathcal E$, agent $j$ is a neighbour of agent $i$.   
We introduce the Laplacian matrix of the graph  denoted $L$ and defined as
\begin{equation}
    l_{ij} = -a_{ij}, \quad \text{for $i \neq j$}, \hspace{0.7cm} l_{ij} = \sum_{k = 1}^N a_{ik}, \quad  \text{for $i=j$}.
\end{equation}
where $l_{ij}$ is the $(i, \, j)$-th entry of $L$.

Our objective is to design an event-triggered distributed control law 
 in order to make the MAS asymptotically reach a consensus, i.e. to
make the consensus manifold $\mathcal D$ defined by
\begin{equation}
    \mathcal D = \left\{\mathbf x  = \text{col}(x_1,\,\dots,\,x_N) \in R^{nN}\, |\, x_1  = \dots = x_N \right\}\!,
\end{equation}
asymptotically stable along the solution of the complete dynamical system \eqref{eq_DynLeader}-\eqref{dyna}. We denote for all $\mathbf x \in \RR^{nN}$ its Euclidean distance to the set $\mathcal D$ by $\left|\mathbf x\right|_{\mathcal D}$.


\subsection{Limitation of communications among the MAS}

To reduce the communications, the agents can communicate only at some particular time instants, that will be defined by the CTC. 
When an agent communicates, it sends its current state value to all its neighbours.
Hence, for all  $j$ in $\mathcal V$, there exists a discrete sequence $t_{j,p}$ in $\RR_+$  which is the $p$-th instant of communication from agent $j$ to its neighbours.

\subsection{Structure of the distributed control law}

Inspired by \cite{garcia2014decentralized}, to each agent $i$ is attached an estimate of the state of its neighbours in the communication graph denoted $\hat x^i_j$ for all $(i,j)$ in $\mathcal E$.
For all $(i,j)$ in $\mathcal E$,
\begin{subequations} 
\label{estimator}
\begin{empheq}[left =\empheqlbrace]{align}
\label{dyna_estimator}
  {{\dot{\hat{x}}}}^i_j(t) &=   f({\hat{x}}^i_j(t),\, \hat\theta),  \forall t \in \left[t_{j,p}, t_{j,(p+1)} \right),\\
  \label{estimator_received}
  {\hat{x}}_j^i(t_{j,p}) &= x_j(t_{j,p}).
\end{empheq}
\end{subequations}
where 
$\hat \theta \in \Omega$ is an a-priori estimate of the vector of unknown parameters $\theta$, assumed to be available to all the agents.
%
We denote (skipping the time dependency of the solution)
$$
\hat {\mathbf x}^i = \text{col}\{\hat x_j^i \},\, \forall\, (i,j)\in\mathcal E 
, \quad
\mathbf x = \text{col}(x_1, \, \dots, \, x_N ).
$$
The distributed control law is then defined as
\begin{equation}
    u_i(t) = \phi_i(\hat {\mathbf x}^i(t)) ,
\end{equation}
where $\phi_i:\RR^{nN}\mapsto\RR^m$ is a 
function in the form\footnote{With a slight abuse of notations this function definition also incorporates agents $j$ which are not neighbours of $i$ since their associate  weights $l_{ij}$ are $0$.}
\begin{equation}
\label{phi}
\phi_i(\mathbf{\hat x}^i)= - \kappa \sum_{j = 1}^N l_{ij}\, \alpha(\hat x_j^i),
\end{equation}
with $\kappa \in \mathbb R$ and $\alpha: \mathbb R^n \rightarrow \mathbb R^m$, a to-be-designed function.
\begin{remark}
To limit the complexity of the estimators and the amount of exchanged information, it has been chosen not to use the control input $u_i$ in the estimators \eqref{estimator}. This corresponds to the open-loop estimation approach described in \cite{survey}.
\end{remark}

The following section is dedicated to selection of the communication instants 
$t_{j,p}$, the function $\alpha$ and the parameter $\kappa$ in order to obtain exponential convergence of the trajectories to the consensus manifold $\mathcal D$.

        

\section{Asymptotic Consensus}
\label{main_results}

\subsection{Additional assumptions}

To reach consensus in the MAS one needs to assume connectivity of the communication graph and a stabilisability property (see for instance \cite{olfati2007consensus}). 
These assumptions are detailed in this subsection.
\subsubsection{Graph connectivity}
\begin{hyp}
\label{hyp_lap}
The graph $\mathcal G$ is connected and undirected.
\end{hyp}
As a consequence, according to  Assumption \ref{hyp_lap}, the Laplacian matrix $L$ associated to the graph $\mathcal G$ can be partitioned as (see \cite{godsil2001algebraic})
\begin{equation}
     L = \begin{pmatrix}l_{11} & L_{12} \\ L_{21} & L_{22}\end{pmatrix}\!,
\end{equation}
where  $L_{22} = L_{22}^\top \in \RR^{{\left(N-1\right)}\times\left(N-1\right)}$ is a positive definite matrix satisfying:
\begin{equation}
\label{L_22}
    L_{22} \geq \mu \,  I,
\end{equation}
with $\mu>0$ and $l_{11}\in \RR$, $L_{21}$ and $L_{12}^\top$ in $\RR^{N-1}$.

\subsubsection{Stabilisability of individual dynamics}
The second assumption is related to a stabilisability property of each agents dynamics. 
\begin{hyp}
\label{hyp_CMF}
There exists a positive definite matrix $P \in \mathbb R^{n\times n}$ and  $q \in \mathbb R^*_+$ such that the following condition holds
\begin{equation}
    \label{CMF}
        \frac{\partial f}{\partial x}(x, \, \theta)^\top  P  + P\,\frac{\partial f}{\partial x}(x, \, \theta)  - \rho\,P\,B\,B^\top P \leq -q\, P,
    \end{equation}
with $\rho>0$ and for all $\theta \in \Omega$ and $x \in \mathbb R^n$. 
\end{hyp}
In the particular case in which $f(x,\theta)=Fx$ for some matrix $F$ (i.e. in the linear framework), equation \eqref{CMF}  is a Riccati equation in the form
$$
        F^\top P  + P\,F  - \rho\,P\,B\,B^\top P \leq -q\, P,
$$
which has been employed in many control design to achieve consensus (see for instance \cite{garcia2014decentralized}).
Hence, \eqref{CMF} is a nonlinear extension of the Ricatti equation. It has been employed to achieve consensus for instance in \cite{andrieu2018some} when no restriction of communication is imposed.

Obtaining a matrix $P$ which satisfies \eqref{CMF} is not an easy task. Indeed, it depends on the coordinates in which is written the system dynamics.
For instance, in the particular case in which the system is feedback linearizable, then this assumption is trivially satisfied. Also, in  \cite{giaccagli2021sufficient2}, the authors suggest several  methods to write equation \eqref{CMF} as a Linear Matrix Inequality (LMI) which might be easily checked. It relies on the decomposition of $f$ as a linear system to which is added an incremental sector bounded nonlinearity. 
This type of assumption has also been employed in \cite{zhang2014fully}.


\subsection{Asymptotic event-triggered consensus}

As in \cite{garcia2014decentralized}, aditionnaly to the estimators \eqref{estimator} of the states of its neighbours, each agent of the MAS also maintains an estimate of its own state corresponding to the current estimate employed by its neighbours to compute their control law. Hence, we introduce for $i=1,\dots,N$,
\begin{subequations} 
\label{estimatori}
\begin{empheq}[left =\empheqlbrace]{align}
\label{dyna_estimatori}
  {{\dot{\hat{x}}}}^i_i(t) &=   f({\hat{x}}^i_i(t),\, \hat\theta),  \forall t \in \left[t_{i,p}, t_{i,(p+1)} \right),\\
  {\hat{x}}_i^i(t_{i,p}) &= x_i(t_{i,p}).
\end{empheq}
\end{subequations}
where $\hat{\theta} \in \Omega$ is the same a priori estimate of the vector of unknown parameters $\theta$ used in \eqref{estimator}.

For $i=1,\dots, N$, let $e_i= x_i - \hat x_i^i$ be the error between $x_i$ and $\hat x_i$, and $r_i = x_i -x_1$ the state error between the agent $i$ and the leader (note that $r_1=0$).



Consider the following CTC defining the sequence of communication instants $t_{i,\,p}$ for all $i =\{1,...,N\} $
\begin{align}\label{SequenceofTime2}
        t_{i,\, p+1} &= \inf\!\left\{\!s\!>t_{i,\,p} \, | \, \delta_i(s) -  \sigma_i\, w_i(s)^\top  \Theta_i \,  w_i(s)> 0 \right\}\!,
\end{align}
with $0< \sigma_i <1$,
\begin{equation}
\label{wi}
     w_i =  \sum_{j=1}^N  l_{ij} \left(\hat{x}_j^i - \hat{x}_i^i \right)\!,
\end{equation}
and
\begin{align}
\label{delta}
     \delta_i &=e_i^\top  S_i \,e_i+  \left| w_i^\top   R_i\, e_i\right|, 
\end{align}
with 
\begin{align}
     R_i &= 2\,\kappa\,P\, B\, B^\top P,\\
     \label{Theta}
     \Theta_i &= 2\,\kappa_2\, \varepsilon\left(1 -2\, l_{ii}\, b_i \right) P\, B\, B^\top P,\\
    \nonumber  S_i &= \left[\left(2\,\kappa\,l_{ii}\, b_i+\frac{2\,\kappa\, l_{ii}}{ b_i}+\kappa_2\,\varepsilon\left( \frac{4\, l_{ii}}{ b_i} \right.\right.\right.\\
    &\hspace{1.5cm} \left. \left.\left.  - N\, M_i\left(\frac{ b_i}{2} + \frac{1}{2\, b_i} \right)\right)\right)P\,B\,B^\top P\right]\!.
\end{align}
where $0<  b_i < \frac{1}{2\,{l_{ii}}}$, $ M_i=\sum_{j=1}^N {l}_{ij}^2$ and $\varepsilon \leq \frac{1}{\lambda_{\max}(L)}$, with $\kappa = \kappa_1 + \kappa_2$, $\kappa_1$ and $\kappa_2$ in $\RR_+$.

Using the following theorem, one can define the distributed event-triggered control law ensuring exponential asymptotic consensus of the MAS. 
\begin{theorem}
\label{theorem}
Suppose Assumption \ref{hyp_lap} and \ref{hyp_CMF} hold.
Then with the distributed control law  $u_i = \phi_i(\mathbf{\hat{x}}^i)$ for $i = 2,\dots,N$, defined in \eqref{estimator}-\eqref{phi}
with $\kappa = \kappa_1 + \kappa_2$, for any $\kappa_1 > \frac{\rho}{2\,\mu}$ and $\kappa_2 > 0$, $\alpha$ defined as 
\begin{equation}
\label{alpha}
    \alpha(x) = B^\top P \,x,
\end{equation}
 $\rho$ satisfying \eqref{CMF}, $\mu$ satisfying \eqref{L_22} 
and  $t_{i,\,p}$ defined according to the CTC \eqref{SequenceofTime2}, 
%
%
%
%
there exist positive constants $k_1$ and $k_2 > 0$ such that for any $\mathbf x(0)$ in $\RR^{Nn}$ and for all $t$ in the time domain of definition of the solution
\begin{equation}
\left|\mathbf x(t)\right|_{\mathcal D} \leq k_1 \exp(-k_2\, t)\left|\mathbf x(0)\right|_{\mathcal D}.
\end{equation}
where $\mathbf x(t)$ denotes the solution of the system \eqref{eq_DynLeader}-\eqref{dyna}.
\end{theorem}

\begin{proof}
The proof of Theorem \ref{theorem} is provided in Appendix \ref{proof_th}.
\end{proof}

\begin{remark}
One can notice that the computation of the CTC and the control law by agent $i$ requires the value of $l_{ij} \,\hat{x}_j$ for all $j$. When $j$ is a neighbour of $i$, then $i$ has access to $\hat{x}_j$. When $j$ is not a neighbour of $i$ then $l_{ij} = 0$, which implies $l_{ij} \,\hat{x}_j=0$.
Therefore the CTC \eqref{SequenceofTime2} and the control law are fully distributed.
%
\end{remark}

\begin{remark}
It has to be noticed that this theorem does not exclude the existence of Zeno behaviour which implies that solutions may not be defined for all positive times. In section \ref{no_zeno}, we propose another CTC to exclude Zeno behaviour.  However the asymptotic convergence is lost and only a practical stabilisation i.e. bounded consensus is obtained.
\end{remark}

\subsection{Practical event-triggered Consensus excluding Zeno behaviour}
\label{no_zeno}
 In order to exclude Zeno behaviour, the former CTC \eqref{SequenceofTime2} is modified by introducing a small positive constant to guarantee a positive inter event time. 
 With this CTC, asymptotic consensus will not be ensured anymore, but the deviation from the consensus can be bounded.
\begin{hyp}
\label{hyp_lip} The function $f$ is $k$-Lipschitz uniformly in $\theta$, i.e., for all $x_i$ and $x_j$ in $\RR^n$
\begin{equation}
    \left|f(x_i, \,\theta) - f(x_j, \,\theta) \right| \leq k \left|x_i - x_j \right|\!.
\end{equation}
Moreover we assume the following property on $f$ holds
\begin{equation}
    \left|f(x_i, \,\hat \theta)- f(x_i, \,\theta) \right| \leq \Delta,
\end{equation}
with $(\theta, \, \hat \theta)$ in $\Omega^2$, for all $x_i \in \mathbb R^n$, and $\Delta \in \RR^+$.
\end{hyp}

The following new CTC is introduced to define the sequence of communication instants $t_{i,\,p}$, for all $i = \{1, \, \dots, \, N\}$ as
\begin{align}\label{SequenceofTimeZeno}
        t_{i,\, p+1} \!= \!\inf\!\left\{\!s\!>\!t_{i,\,p} \, | \, \delta_i(s) \!- \! \sigma_i\, w_i(s)^{\!\top}  \Theta_i \,  w_i(s) \!-\! \xi \!> \!0 \right\}\!,
\end{align}
with $0< \sigma_i <1$,  $\xi>0$,
\begin{equation}
     w_i = \sum_{j=1}^N  l_{ij} \left(\hat{x}_i^i - \hat{x}_j^i \right),
\end{equation}
with $ \delta_i$ defined in \eqref{delta} and $ \Theta_i$ defined in \eqref{Theta}.

Using the following theorem, one can define the distributed-event triggered control law ensuring practical consensus of the MAS and absence of Zeno behaviour. 
\begin{theorem}
\label{th_no_zeno}
Suppose Assumptions \ref{hyp_lap},  \ref{hyp_CMF} and \ref{hyp_lip} hold.
Then with the distributed control law  $u_i = \phi_i(\mathbf{\hat{x}}^i)$ for $i = 2,\dots,N$, defined in \eqref{estimator}-\eqref{phi}
with $\kappa = \kappa_1 + \kappa_2$, for any $\kappa_1 > \frac{\rho}{2\,\mu}$ and $\kappa_2 > 0$, $\alpha$ defined in \eqref{alpha}, $\rho$ satisfying \eqref{CMF}, $\mu$ satisfying \eqref{L_22} 
and  $t_{i,\,p}$ defined according to the CTC \eqref{SequenceofTimeZeno}
%
%
%
%
then 
for all $\mathbf x(0)$ in $\RR^{Nn}$ the solution $\mathbf x(t)$  of the system \eqref{dyna} is defined for all positive time instants and
\begin{equation}
    \lim_{t \rightarrow \infty} \left|\mathbf x(t) \right|_{\mathcal D}^2 \leq \frac{N\,\xi}{q\, \lambda_{\min}(P)},
\end{equation}
with $q$ respecting \eqref{CMF}.
Moreover, there is no Zeno behaviour, and the inter event-time of the solution is bounded by $\tau_i$
\begin{align}
    \tau_i &= \frac{1}{k} \log\left( \sqrt{\frac{\xi}{c_1}  +\frac{c_2^2}{4\,c_1^2}}+1 - \frac{c_2}{2\,c_1}\right)\!,
\end{align}
with $k$ the Lipschitz constant of $f$, $c_1 = \left|S_i \right| \nu^2$ and $c_2 =w_{i_{\max}} \left|R_i\right| \nu$,
where $w_{i_{\max}} = \max_t \left| w_i(t)\right|$ and 
$
    \nu = \frac{\kappa \left|B\, B^\top P \right|\left[ w_{i_{\max}}+ \Delta\right] }{k}.
$
\end{theorem}
\begin{proof}
The proof of Theorem \ref{th_no_zeno} is provided in Appendix \ref{proof_th_no_zeno}
\end{proof}

\section{Simulation examples}
\label{result}
Consider a MAS of $N=5$ agents with a communication graph illustrated in Figure \ref{figure_graph}. 
\pagebreak
\tikzstyle{etat} = [draw, circle]
\begin{center}
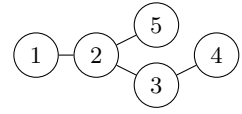

    \centering
    \begin{tikzpicture}[scale = 0.8]
        \node[etat] (1) at (0,0) {$1$};
        \node[etat] (2) at (1,0) {$2$};
        \node[etat] (3) at (2,-0.5) {$3$};
        \node[etat] (5) at (2,0.5) {$5$};
        \node[etat] (4) at (3,0) {$4$};
        
        \draw[-,>=latex] (1) to (2) ;
        \draw[-,>=latex] (2) to (3);
        \draw[-,>=latex] (2) to (5);
        \draw[-,>=latex] (3) to (4);
    \end{tikzpicture} 
    \captionof{figure}{Network Structure}
    \label{figure_graph}
\end{center}
The function $f$ and the matrix $B$ defining the dynamics of the agents are chosen as
\begin{equation}
\label{dynareel}
    f(z, \, \theta) \!=\! \begin{pmatrix}z_2+\theta\,\cos(z_2)\\-z_1+\theta\cos(z_2)+\theta^2\!\cos(z_1)\sin(z_1) \end{pmatrix}\!\!, 
    B = \begin{pmatrix}0\\-1 \end{pmatrix}\!,
\end{equation}
with $z = \left[ z_1 \, z_2 \right]^\top$.
For the simulations, the unknown parameter is set to $\theta = 0.5$ and the initial condition of the MAS is 
$$\mathbf x(0) = \begin{pmatrix}
    (0.95 &  0.63)^\top\\
    (-0.70 & -0.73)^\top \\
    (-0.33& -0.54)^\top\\
   (-0.25 &0.02)^\top \\
   (0.86 & 0.01)^\top
\end{pmatrix}\!,$$

Using \eqref{dynareel} with $\theta=0.5$, a matrix $P$ satisfying condition \eqref{CMF} with constants $\rho = 0.02$, $q=1$ is 
\begin{equation}
    P =
     \begin{pmatrix}
      5& 2 \\ 2 & 1
    \end{pmatrix}\!,
\end{equation}

The other parameters used either in the CTC or control law are chosen as $\kappa_2 = 5$, $ \varepsilon = \frac{1}{\lambda_{\max}(L)}$, $b_i = (5\, l_{ii})^{-1}$, $\kappa_1 = 0.1$, $\sigma _1 = 0.8$, and for all $i \in\{2,\,\dots,\, 5\}$, $\sigma_i = 0.9$. 

A sampling period of $10^{-2}s$ has been chosen for all the numerical simulations. 
Different simulations have been performed to illustrate the resulting performances of the proposed approach in terms of reduction of communication. Analysis of the robustness to the a priori value $\hat{\theta}$ is also presented as well as a comparison of efficiency of the two proposed CTCs \eqref{SequenceofTime2} and \eqref{SequenceofTimeZeno}.

Figure \ref{fig_1} presents the time evolution of the state of each agent of the MAS and the communication instants, with the CTC \eqref{SequenceofTime2} 
using  
$\hat\theta = 0.40$, for which \eqref{CMF} is satisfied with the chosen $P$. The simulation duration $T_{sim}$ is $10 \, s$.


\begin{center}
    \includegraphics[width =\linewidth]{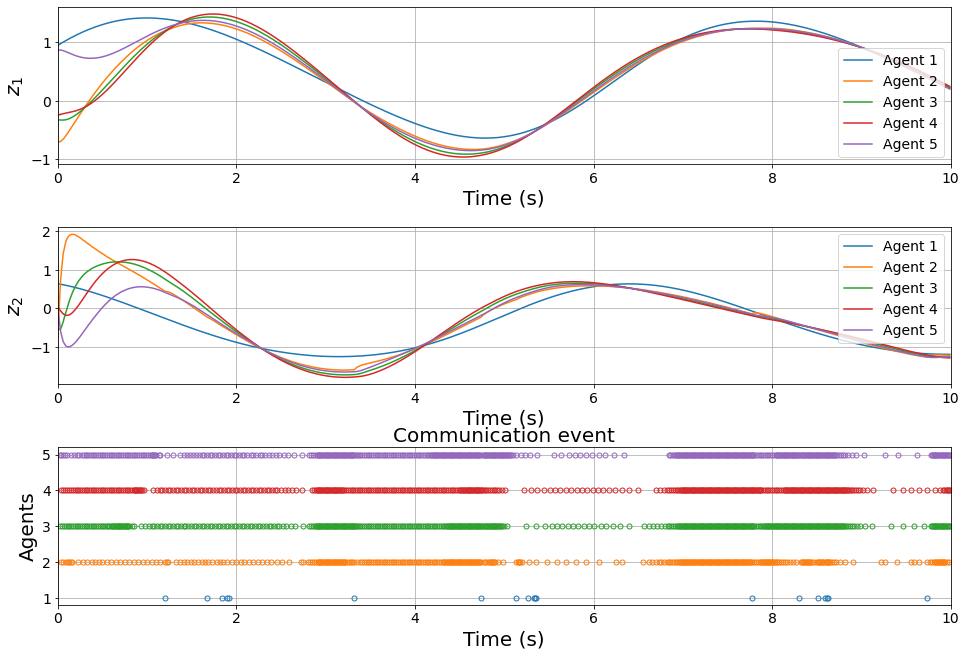}
    \captionof{figure}{\footnotesize Time evolution of the components of state for each agent ({\it top}: first component, {\it middle}: second component), and communication instants for each agent ({\it bottom}). Case with the CTC \eqref{SequenceofTime2}, $\theta=0.50$ and $\hat \theta=0.40$.}
    \label{fig_1}
\end{center}
It can be observed that the consensus is reached with limited communications. The CTC is mostly triggered by the errors due to neglecting the control input $u_i$ in the estimators. Thus, very few communications are triggered by the leader (agent 1) which is not controlled. 

 The results presented in Figure \ref{fig_2} are obtained for $\hat \theta= 0.35$ corresponding to a larger uncertainty but still satisfying \eqref{CMF}. The CTC  \eqref{SequenceofTime2} is also used in this example and other parameters' values remain unchanged. As can be seen, consensus is also reached in this case by the agents with non-periodic communications. 
\begin{center}
    \includegraphics[width =\linewidth]{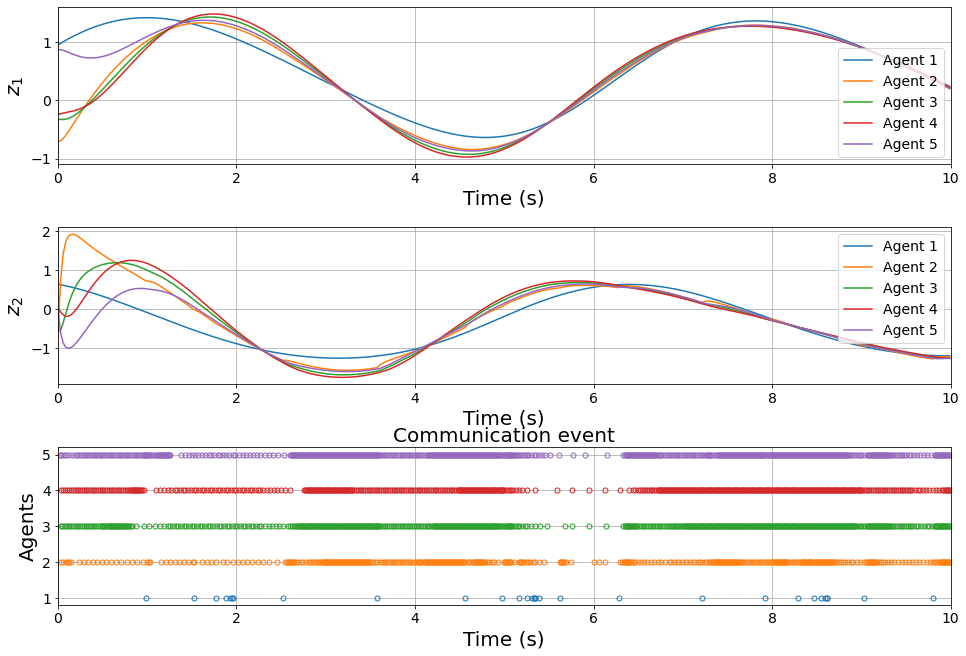}
    \captionof{figure}{\footnotesize Time evolution of the components of state for each agent ({\it top}: first component, {\it middle}: second component), and communication instants for each agent ({\it bottom}). Case with the CTC \eqref{SequenceofTime2}, $\theta= 0.50$ and $\hat \theta= 0.35$.}
\label{fig_2}
\end{center}
In both cases, non-periodic communications are obtained, as triggered by the proposed CTC \eqref{SequenceofTime2}, and convergence to an asymptotic consensus is observed. 
In order to compare the resulting performances, we introduce a performance index $\Gamma$ to qualify the evolution of the consensus error and the reduction of communications from the CTC. $\Gamma$ is defined as
\begin{equation}
    \Gamma = \frac{1}{N}\sum_{k = 0}^K \left[ \sum_{i=1}^N \left|r_i(k)\right|_2 +\chi\, v_i(k)\right] 
\end{equation}
where $K$ is the number of simulated instants, $v_i(k)$ is defined as
\begin{subequations} 
\begin{empheq}[left ={v_i(k) \!=\!\empheqlbrace}]{align}
  &1  &&\text{if the CTC is verified},\\
 &0  &&\text{otherwise}.
\end{empheq}
\end{subequations}
and $\chi$ is a tuning parameter that allows to balance between consensus accuracy and communication reduction.
Comparisons of the performances for the two values of $\hat{\theta}$ are presented in Table \ref{tab}, for $\chi$ equals to $0.1$ and for two values of $T_{sim}$, respectively $10$ and $30$ seconds. As expected, it can be observed that a larger uncertainty ($\hat{\theta}=0.35$) results in some discrepancy of the performances, however limited. 
In order to compare the efficiency of the two CTCs, simulations have been performed with CTC \eqref{SequenceofTimeZeno} and the two values of $\hat{\theta}$.
The initial condition and the other parameters remain unchanged. The parameter $\xi$ has been selected equal to 20. This value has been chosen to illustrate the bounded consensus in the simulation. Note that better performances in terms of consensus accuracy could be obtained for smaller values of $\xi$.

Figure \ref{fig_3} presents the corresponding  time evolution of the state of each agent of the MAS and the communication instants.
\begin{center}
    \includegraphics[width =\linewidth]{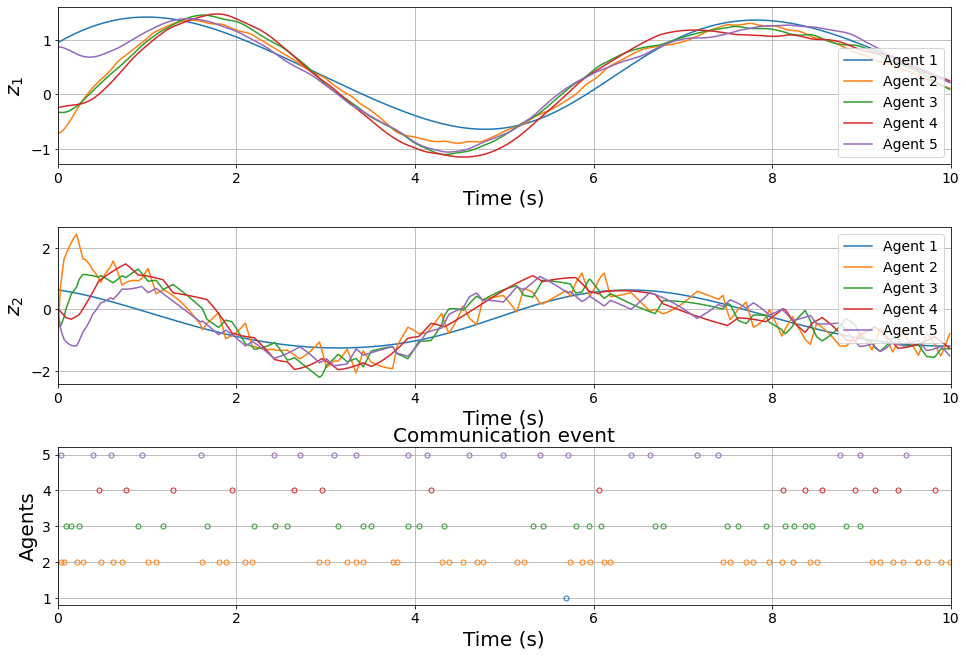}
    \captionof{figure}{\footnotesize Time evolution of the components of state for each agent ({\it top}: first component, {\it middle}: second component), and communication instants for each agent ({\it bottom}). Case with  the CTC \eqref{SequenceofTimeZeno}, $\theta= 0.50$ and $\hat \theta= 0.40$.}
\label{fig_3}
\end{center}
As expected, with this CTC and value of $\xi$, the communications are drastically reduced at the cost of losing asymptotic convergence to the consensus. Bounded variations around the consensus manifold appear clearly on the figure.
The corresponding values of performance index $\Gamma$ for these examples are also presented in Table \ref{tab}.

\begin{table}[h]
    \centering
    \begin{tabular}[H]{ |c || c | c | c | c | }
     \hline
     \multicolumn{5}{|c|}{Values of $\Gamma$} \\
     \hline
     & \multicolumn{2}{|c|}{CTC \eqref{SequenceofTime2}} & \multicolumn{2}{|c|}{CTC \eqref{SequenceofTimeZeno}}\\
     \hline
     $T_{sim}$ & $\hat \theta = 0.40$& $\hat \theta = 0.35$ & $\hat \theta = 0.40$ & $\hat \theta = 0.35$\\
     \hline
     10 s & $387$ & $393$ & $343$ & $395$\\
     \hline
     30 s & $727$ & $764$ & $843$ & $843$\\
     \hline
    \end{tabular}
    \caption{Comparison of $\Gamma$ }
    \label{tab}
\end{table}

From Table \ref{tab}, it can be seen that, for a shorter simulation, the CTC \eqref{SequenceofTimeZeno} is better than the CTC \eqref{SequenceofTime2} according the performance index $\Gamma$. This result is consistent since the  CTC \eqref{SequenceofTimeZeno} triggers less communications than the CTC \eqref{SequenceofTime2}. Nevertheless since the asymptotic consensus is not reached with the CTC \eqref{SequenceofTimeZeno}, for a longer simulation, the CTC \eqref{SequenceofTime2} proves to be more efficient. In order to explicit this phenomena,  the quantity $\frac{1}{N}\sum_{k=1}^K\sum_{i=1}^N \left|r_i(k)\right|_2$, which represents the quality of the consensus and the quantity $\sum_{i=1}^K\sum_{k=1}^N v_i(k)$ which represents the number of communications in the MAS have been evaluated separately. 

 Tables \ref{tab_2} and \ref{tab_3} present the different values of these quantities for all the simulated examples. Clearly, the CTC \eqref{SequenceofTime2} allows to reach the consensus with more accuracy than the CTC \eqref{SequenceofTimeZeno} at the cost of increasing the number of communications.
\begin{table}[h]
    \centering
    \begin{tabular}[H]{ |c || c | c | c | c | }
     \hline
     \multicolumn{5}{|c|}{Values of $\frac{1}{N}\sum_{k=1}^K\sum_{i=1}^N \left|r_i(k)\right|_2$} \\
     \hline
     & \multicolumn{2}{|c|}{CTC \eqref{SequenceofTime2}} & \multicolumn{2}{|c|}{CTC \eqref{SequenceofTimeZeno}}\\
     \hline
     $T_{sim}$ & $\hat \theta = 0.40$& $\hat \theta = 0.35$ & $\hat \theta = 0.40$ & $\hat \theta = 0.35$\\
     \hline
     10 s & $311$ & $300$ & $420$ & $393$\\
     \hline
     30 s & $409$ & $388$ & $829$ & $830$\\
     \hline
    \end{tabular}
    \caption{Comparison of $\frac{1}{N}\sum_{k=1}^K\sum_{i=1}^N \left|r_i(k)\right|_2$ }
    \label{tab_2}
\end{table}

\begin{table}[h]
    \centering
    \begin{tabular}[H]{ |c || c | c | c | c | }
     \hline
     \multicolumn{5}{|c|}{Values of $\sum_{k=1}^K\sum_{i=1}^N v_i(k)$} \\
     \hline
     & \multicolumn{2}{|c|}{CTC \eqref{SequenceofTime2}} & \multicolumn{2}{|c|}{CTC \eqref{SequenceofTimeZeno}}\\
     \hline
     $T_{sim}$ & $\hat \theta = 0.40$& $\hat \theta = 0.35$ & $\hat \theta = 0.40$ & $\hat \theta = 0.35$\\
     \hline
     10 s & $3800$ & $4650$ & $250$ & $200$\\
     \hline
     30 s & $15850$ & $18800$ & $700$ & $650$\\
     \hline
    \end{tabular}
    \caption{Comparison of number of communications in the MAS }
    \label{tab_3}
\end{table}


\pagebreak
\section{Conclusion}
In this paper,  a new distributed event-triggered control law is proposed for the leader-follower consensus problem of a MAS described with uncertain control-affine nonlinear dynamics.
It is based on two novel Communication Triggering Conditions (CTCs): the first one allowing to achieve an asymptotic consensus but without any guarantee on Zeno behaviour and the second one that exclude Zeno behaviour but results only  in a bounded consensus. Simulations have been presented to compare the performances of the two approaches. Robustness to uncertainty is also illustrated.  Future work will consider agents with more generic nonlinear dynamics and state perturbations and MAS composed by heterogeneous agents.



\printbibliography

\section{Appendix}
\subsection{Proof of the Theorem \ref{theorem}}
\label{proof_th}
\begin{proof}
   Since \eqref{estimator} and \eqref{estimatori} follow the same dynamical system, it yields for all $(j,i)$ in $\mathcal E$
$$
\hat x_j^i(t)=\hat x_j^j(t) .
$$
For $i=2,\dots,N$, the $r_i$'s dynamics are given by:
\begin{equation}
    \dot{r_i} = f(r_i + x_1, \, \theta) - f(x_1, \, \theta) - \kappa \sum_{j = 1}^N l_{ij}\,B\,\alpha(x_1 + r_j-e_j).
\end{equation}
Using the fact that for $i=2,\dots,N$:
\begin{equation}
    \sum_{j = 1}^N l_{ij}\,B\,\alpha(x_1) = B\,\alpha(x_1)\sum_{j = 1}^N l_{ij} = 0,
\end{equation}
we can write:
\begin{align}
    \nonumber \dot{r_i} &= f(r_i + x_1, \, \theta) - f(x_1, \, \theta) \\
    \label{r_point}
    &\hspace{1.5cm}- \kappa \sum_{j = 1}^N l_{ij}\,B\left(\alpha(x_1 + r_j-e_j) - \alpha(x_1)\right).
\end{align}

Consider the following 
function which is defined on the time domain of definition of the solution,
\begin{equation}
    V_i(t) = r_i(t)^\top P \, r_i(t)\ , \ i=2,\dots,N.
\end{equation}
Using the relation \eqref{r_point}, the time-derivative of $(V_i)'s$ can be computed as (forgetting the time dependency of the state variables):
\begin{align}
    \nonumber \dot{V}_i &= 2\,r_i^\top P\,\left(f(r_i+x_1, \, \theta)-f(x_1, \, \theta)\right) \\
    &\hspace{1.3cm}- 2\, \kappa \,r_i^\top P\,B \sum_{j=1}^N l_{ij} \left(\alpha(r_j+x_1) - \alpha(x_1) \right)\!.
\end{align}
For fixed $x_1, r_i, \theta, e_i$, let us define $F_i(s): \mathbb R \rightarrow \mathbb R^n$ and $A_i(s): \mathbb R \rightarrow \mathbb R^n$ as: 
\begin{align}
    F_i(s) &= f(x_1 + r_i\,s, \, \theta), \\
    A_i(s) &= \alpha(x_1 + \left(r_i - e_i\right)s).
\end{align}
Note that $\left(f_i(r_i+x_1, \, \theta)-f_i(x_1, \, \theta)\right) = \int_0^1\frac{d\, F_i(s)}{d s} ds$. Hence
\begin{align}
    \nonumber \dot{V}_i &= \!\int_0^1\!\! r_i^\top P\,\frac{d F_i(s)}{d s} \!+\! \frac{d F_i(s)}{d s}^{\!\top}\!\!\! P \, r_i \!-\! 2\, \kappa \,r_i^\top \!PB\! \sum_{j=1}^N \!l_{ij} \frac{d A_j(s)}{d s}ds.
\end{align}
One can notice that:
\begin{align}
    \frac{d\, F_i(s)}{d s} &= \frac{\partial f}{\partial x_i}(x_1 + r_i\,s, \, \theta)\,r_i,\\
    \frac{d\, A_i(s)}{d s} &= \frac{\partial \alpha}{\partial x_i}(x_1 + \left(r_i - e_i\right)s)\,\left(r_i -e_i \right).
\end{align}
Therefore, for $i=2,\dots,N$,
\begin{align}
    \nonumber\dot{V}_i =& \int_0^1 \!r_i^\top \!\!\left(\! P\,\frac{\partial f}{\partial x_i}(x_1 + r_i\,s, \,\theta) + \frac{\partial f}{\partial x_i}(x_1 + r_i\,s, \, \theta)^\top \!P\!\right) r_i \,ds\\
    \nonumber &- 2\, \kappa \,r_i^\top P\,B \sum_{j=1}^N l_{ij} B^\top P \,\left(r_j -e_j \right).
\end{align}
Consider $v$ given as $ v = \left[ (B^\top P \, r_2)^\top, ..., ( B^\top P \, r_N)^\top \right]^\top$.
Define $V(t) = \sum_{i=2}^N V_i(t)$. 
\begin{align}
        \nonumber &\dot V= \int_0^1 \sum_{i=2}^N r_i^\top \He\left( P\,\frac{\partial f}{\partial x_i}(x_1 + r_i\,s, \, \theta)
        \right)\, r_i\, ds\\
        &+ 2\kappa \sum_{i=2}^N\sum_{j=1}^N l_{ij}\, r_i^\top PB B^\top\! P \,e_j- 2 \kappa v^\top (L_{22}\otimes I_{n}) \, v.\label{point_d_arret}
\end{align}
Using \eqref{L_22} and the condition \eqref{CMF} and selecting any $\kappa_1\geq \frac{\rho}{2\,\mu}$, we have:
\begin{align*}
        \nonumber \dot V \leq&  \sum_{i=2}^N -q\, r_i^\top P r_i +2\,\kappa \sum_{i=2}^N\sum_{j=1}^N l_{ij}\, r_i^\top P\,B \, B^\top P \,e_j\\
        \nonumber & - 2\, \kappa_2\sum_{i=2}^N \sum_{j=1}^N l_{ij} \,r_i^\top \, P\, B\, B^\top P\,r_j.
\end{align*}
We denote $\hat x = \text{col}(\hat x_1, \, \dots, \,\hat x_N)$, $r= \text{col}(r_1, \, \dots,\, r_N)$ and $e= \text{col}(e_1, \, \dots,\, e_N)$. Since $r_1 = 0$, one has:
\begin{align}
    \nonumber \sum_{i=2}^N \sum_{j=1}^N l_{ij} \,&r_i^\top\, P\, B\, B^\top P\,r_j =  r^\top \!\left( L \otimes \left( P\, B\, B^\top P\right) \right) r,\\
    \nonumber \sum_{i=2}^N \sum_{j=1}^N l_{ij} \,&r_i^\top\, P\, B\, B^\top P\,e_j =  r^\top \!\left( L \otimes \left( P\, B\, B^\top P\right) \right) e.
\end{align}
Define $\varepsilon \leq \frac{1}{\lambda_{\max}(L)}$. Since $L$ is a symmetric semi-definite postive matrix, one has
\begin{equation}
     - L \leq -\varepsilon\, L^2 \leq -\varepsilon\, L^\top L.
\end{equation}
Moreover $P\, B \, B^\top P$ is a semi-definite positive matrix, so
\begin{equation}
     - L \otimes \left( P\, B\, B^\top P\right) \leq -\varepsilon\, \left( L^\top L \right) \otimes \left( P\, B\, B^\top P\right)\!.
\end{equation}
Thus:
\begin{align}
    \nonumber &-\!\! \sum_{i=2}^N \sum_{j=1}^N l_{ij} \,r_i^\top \! P B B^\top\! P\,r_j
    \leq - \varepsilon \,r^\top \!\!\left(\!\left(L^\top \!L\right) \otimes\left( P B B^\top\! P\right) \!\right) r,\\
    \label{50}
    &\, \leq -\varepsilon \!\sum_{i=1}^N\!\left[\sum_{j=1}^N l_{ij} (r_j-r_i)^\top \!P B B^\top \!P\sum_{k=1}^N  l_{ik} (r_k-r_i) \right]\!\!.
\end{align}
One can notice that
\begin{align}
    \nonumber\left( L \otimes I_N\right) r &= \sum_{j=1}^N l_{ij} \, (r_j-r_i),\\
    \nonumber &= \sum_{j=1}^N l_{ij} \left[(\hat{x}_j - \hat{x}_i)  +(e_j  - e_i)\right],\\
    \label{51}
    &= \left(L \otimes I_n\right)\hat{x} + \left( L \otimes I_n\right) e.
\end{align}
Since $L$ is symmetric and $\sum^N_{j=1} l_{ij} =0$, we have
\begin{align}
    \nonumber r^\top\!\!&\left(L\!\otimes \!\left(PBB^\top \!P\right)\! \right)e = \sum_{i=1}^N\sum_{j=1}^N l_{ij}\, r_i^\top P\,B \, B^\top P \,e_j, \\
    \nonumber &\hspace{1.5cm}= \sum_{j=1}^N\sum_{i=1}^N l_{ji}\, r_i^\top P\,B \, B^\top P \,e_j, \\
    \nonumber &\hspace{1.5cm}= \sum_{i=1}^N\sum_{j=1}^N l_{ij}\, r_j^\top P\,B \, B^\top P \,e_i, \\
    \nonumber &\hspace{1.5cm}= \sum_{i=1}^N\sum_{j=1}^N l_{ij}\, (r_j - r_i)^\top P\,B \, B^\top P \,e_i, \\
    \label{52}
    &\hspace{1.5cm}=  \left(\left( L \otimes I_n\right) r\right)^\top\! \left( I_N \!\otimes \! \left(P BB^\top\! P\right)\right) e.
\end{align}
Therefore, using \eqref{50}, \eqref{51} and \eqref{52}, we have
\begin{align}
    &\nonumber \dot V\leq  \sum_{i=1}^N\left[-  q\, r_i^\top P\, r_i +2\,\kappa \,\sum_{j=1}^N l_{ij} \left( \hat x_j - \hat x_i\right)^\top \! P B B^\top \! P \, e_i \right. \\
    \nonumber &  +2\,\kappa \!\sum_{j=1}^N \! l_{ij} \,e_j^\top  P B B^\top \!\! P \, e_i- 2\,\kappa_2\,\varepsilon \!\sum_{j=1}^N \! l_{ij}\, e_j^{\top}  P B B^\top \!\! P \!\sum_{k=1}^N \! l_{ik} \,e_k\\
    \nonumber&- 2\,\kappa_2\,\varepsilon \left(\sum_{j=1}^N l_{ij} \left( \hat x_j - \hat x_i\right)\right)^{\!\top}\!\!\! P B B^\top \! P \left(\sum_{k=1}^N  l_{ik} \left( \hat x_k - \hat x_i\right)\right) \\
    &\left.- 4\,\kappa_2  \,\varepsilon \left(\sum_{j=1}^N  l_{ij} \left( \hat x_j - \hat x_i\right)\right)^{\!\top}\!\!\!  P B B^\top \! P \sum_{k=1}^N  l_{ik} \,e_k\right]\!.
\end{align}
Since $L$ is symmetric, we have
\begin{align}
\label{55}
    \sum_{i=1}^N\sum_{j=1}^N\left| l_{ij}\right| e_j^\top P\,B \, B^\top P \,e_j = \sum_{i=1}^N 2\,l_{ii}\, e_i^\top P\,B \, B^\top P \,e_i.
\end{align}
Define $ w_i = \sum_{j=1}^N l_{ij} \left(\hat{x}_j - \hat{x}_i \right)$. Note that $ w_1 \neq 0$. Thus using the inequality $\left| x^\top y\right| \leq \frac{ b_i}{2}x^\top x+\frac{1}{2\, b_i}y^\top y$ with $ b_i> 0$, and \eqref{55}, we have
\begin{align}
    \nonumber &\dot V\leq \sum_{i=1}^N\!\left[\vphantom{\sum_j^N} \left|2\,\kappa \, w_i^\top P B B^\top\! P \, e_i\right|- 2\,\kappa_2\, \varepsilon \, w_i^\top P B B^\top\! P \, w_i\right.\\
    \nonumber &+\kappa \left(2\, l_{ii}\, b_i\,e_i^\top P\, B\,B^\top P\, e_i + \sum_{j=1}^N\frac{\left| l_{ij}\right|}{ b_i}\, e_i^\top P\,B \, B^\top\! P \,e_i\right)  \\
    \nonumber& + 2\,\kappa_2\,\varepsilon \!\left(2\, l_{ii}\, b_i\, w_i^\top P B B^\top\! P\,  w_i+ \sum_{j=1}^N\frac{\left| l_{ij}\right|}{ b_i}\, e_i^\top P B B^\top \!P \,e_i\right)\\
    &\left. +\kappa_2\, \varepsilon \,N\!\left(\frac{ b_i}{2}+\frac{1}{2\, b_i}\right)\!\sum_{j=1}^N \!l_{ij}^{\,2}\, e_i^\top PB B^\top\!\! P \,e_i - q\, r_i^\top P\, r_i\right]\!\!.
\end{align}
Define $ M_i= \sum_{j=1}^N {l}_{ij}^{\,2}$. Therefore, we obtain
\begin{align}
    \nonumber \dot V&\leq \sum_{i=1}^N\left[- q\, r_i^\top P\, r_i+  \left|w_i^\top   R_i\, e_i \right|
    -  w_i^\top  \Theta_i \, w_i
+ e_i^\top  S_i \,e_i \right]\!,
\end{align}
with
\begin{align}
     R_i &= 2\,\kappa\,P\, B\, B^\top P,\\
     \Theta_i &= 2\,\kappa_2\, \varepsilon\left(1 -2\, l_{ii}\, b_i \right) P\, B\, B^\top P,\\
    \nonumber  S_i &= \left[\left(2\,\kappa\,l_{ii}\, b_i+\frac{2\,\kappa\, l_{ii}}{ b_i}+\kappa_2\,\varepsilon\left( \frac{4\, l_{ii}}{ b_i} \right.\right.\right.\\
    &\hspace{1.5cm} \left. \left.\left.  - N\, M_i\left(\frac{ b_i}{2} + \frac{1}{2\, b_i} \right)\right)\right)P\,B\,B^\top\! P\right]\!.
\end{align}
and if we choose $0 \leq  b_i \leq \frac{1}{2\,l_{ii}}$, then $ \Theta_i \geq 0$.
Define $ \delta_i =e_i^\top  S_i \,e_i+ \left| w_i^\top   R_i\, e_i\right|$. Therefore
\begin{align}
        \dot V&\leq \sum_{i=1}^N\left[\vphantom{\frac{2}{b_i}} -q\, r_i^\top P\, r_i -   w_i^\top \Theta_i \, w_i + \delta_i \right]\!.
\end{align}
If the CTC \eqref{SequenceofTime2} is respected then the error is reset to zero, that is, $e_i(t_{i,\, k}) = 0 \implies  \delta_i(t_{i,\,k}) = 0$. It yields, 
\begin{align}
     \delta_i \leq  \sigma_i\, w_i^\top  \Theta_i \,  w_i.
\end{align}
Thus
\begin{equation}
\label{point_d_arret_Zeno}
        \dot V \leq \sum_{i=1}^N \left[\vphantom{\frac{2}{b_i}} -q\, r_i^\top P\, r_i + \left( \sigma_i - 1\right) w_i^\top  \Theta_i \,  w_i \right]\!.
\end{equation}
Since $\sigma_i -1 < 0$, and $ \Theta_i \geq 0,$, then $\left( \sigma_i -1\right) w_i^\top  \Theta_i \,  w_i \leq 0$.
Finally we have
\begin{equation}
        \dot V \leq -q\,V.
\end{equation}
Hence, on the time domain of existence of the solution, one gets,
$$
V(t) \leq \exp(-qt) V(0)\ .
$$
The matrix $P$ being positive definite, the result may be easily obtained.
\end{proof}

\subsection{Proof of the Theorem \ref{th_no_zeno}}
\label{proof_th_no_zeno}
\begin{proof}
   Starting from \eqref{point_d_arret_Zeno}, we have
\begin{align}
        \nonumber \dot V&\leq \sum_{i=1}^N\left[\vphantom{\frac{2}{b_i}} - q\, r_i^\top P\, r_i-   w_i^\top  \Theta_i \, w_i + \delta_i \right]\!.
\end{align}
According to the CTC \eqref{SequenceofTimeZeno}, $ \delta_i < \sigma_i \, w_i^\top \Theta \,  w_i + \xi$.
Hence, one has
\begin{align}
        \nonumber \dot V&\leq \sum_{i=1}^N\left[\vphantom{\frac{2}{b_i}} - q\, r_i^\top P\, r_i -  \left(1- \sigma_i\right) w_i^\top  \Theta_i \, w_i + \xi \right]\!,\\
        \label{integrating}
        \dot V&\leq N\,\xi-q \,V.
\end{align}
By integrating \eqref{integrating}, we obtained
\begin{align}
    \nonumber V(t) &\leq e^{-q\,t}\,V(0)  + N\,\xi\int_0^t e^{-q\,\left(t - \tau \right)}d\tau,\\
    &\leq \left(V(0) - \frac{N\,\xi}{q\,}\right)e^{-q\,t} +\frac{N\,\xi}{q}.
\end{align}
We recall that $V = \sum_{i=2}^N r_i ^\top P \,r_i$, thus, as $P$ is definite positive, we have
\begin{align*}
    \nonumber \lambda_\min(P) \left|r_i(t) \right|^2 &\leq r_i(t)^\top P \, r_i(t), \\
    &\leq \left(V(0) - \frac{N\,\xi}{q}\right)e^{-q\,t} +\frac{N\,\xi}{q}.
\end{align*}
Therefore, we have $\lim_{t \rightarrow \infty} \left|r_i(t)\right|^2\leq \frac{\left(N-1 \right)\xi}{q\, \lambda_{\min}(P)}$.\\
Note that, we have
\begin{align}
    \left|\mathbf x \right|_{\mathcal D}^2 = \min_{z \in \RR^n}\sum_{i=1}^N \left|z - x_i \right|^2\leq \sum_{i=1}^N \left|x_1 -x_i \right|^2 = \left|r \right|^2
\end{align}
Thus $\lim_{t \rightarrow \infty}\left|\mathbf x(t) \right|_{\mathcal D}^2 \leq \frac{N\,\xi}{q\, \lambda_{\min}(P)}$.
In this second part, we want to show that the inter event time between two communications can be bounded.
Consider the error's dynamics
\begin{align}
    \nonumber \frac{d}{dt}\left|e_i\right| &= \left(e_i^\top e_i \right)^{-\frac{1}{2}}e_i^\top \dot{e}_i,\\
    \nonumber &= \left(e_i^\top e_i \right)^{-\frac{1}{2}}e_i^\top\!\!\left(f(x_i, \, \theta) - \kappa \,B B^\top\! P \, w_i - f(\hat{x}_i, \, \hat \theta)\right),\\
    \nonumber &\leq \left(e_i^\top e_i \right)^{-\frac{1}{2}}\left|e_i\right|\left|f(x_i, \, \theta) -  f(\hat{x}_i, \, \theta)\right| \\
    \nonumber &\hspace{2cm}+ \left(e_i^\top e_i \right)^{-\frac{1}{2}}\left|e_i\right|\left(\left|\kappa \,B\, B^\top P \,  w_i\right| + \Delta\right)\!,\\
    \label{equadiff_ei}
    &\leq k\left|e_i\right|+ \kappa \,|B\, B^\top P |\, ( w_{i_{\max}} + \Delta).
\end{align}
where $w_{i_{\max}} = \max_t \left| w_i(t)\right|$ and $k$ the Lipschitz constant of $f$.
Define $ w = \text{col}( w_1, \dots, \,  w_N )$. By definition and using \eqref{51} we have $ w = \left(L \otimes I \right) \left( r - e\right)$. According to the first part of the proof, we show that $V$ is bounded so as $\left|r \right|$, thus according to the CTC \eqref{SequenceofTimeZeno}, $\left| e\right|$ is also bounded. To conclude $w_{i_{\max}}$ exists and it is finite.
By integrating \eqref{equadiff_ei} with $\left|e_i(0) \right| = 0$, one obtains
\begin{equation}
    \left|e_i\right| \leq \frac{\kappa \left|B\, B^\top P \right|\left( w_{i_{\max}}+ \Delta\right) }{k} \left[e^{k\,t}-1 \right]\!.
\end{equation}
Define $\nu = \frac{\kappa \left|B B^\top\! P \right|\left( w_{i_{\max}}+ \Delta\right) }{k}$. Thus we can bound the evolution of $\delta_i$ by
\begin{align}
    \nonumber  \delta_i &\leq \left|S_i \right| \nu^2 \left[e^{k\,t}-1 \right]^2 + w_{i_{\max}} \left|R_i\right|\nu\left[e^{k\,t}-1 \right] \!,\\
    &\leq c_1\, e^{2\,k\,t} -\left(2 \,c_1 - c_2\right)\,e^{k\,t} + c_1 - c_2,
\end{align}
with $c_1 = \left|S_i \right| \nu^2$ and $c_2 =w_{i_{\max}} \left|R_i\right| \nu$.

One can notice that the time it takes for $\bar\delta_i$ to grow from $0$ to $ w_i ^\top \Theta_i \,  w_i+ \xi$ is not smaller that the time need for $\bar\delta_i$ to grow from $0$ to $\xi$.
Define by $\tau_i$ the lower bound between two events. It can be obtained by solving the following equation
\begin{align}
\label{equation_tau}
    \nonumber \xi &= c_1\, e^{2\,k\,\tau_i} -\left(2 \,c_1 - c_2\right)\,e^{k\,\tau_i} + c_1 - c_2, \\
    \nonumber \xi &= \left(\sqrt{c_1}\, e^{k\,\tau_i} - \left(\frac{1}{2\,\sqrt{c_1}}\left(2 \,c_1 - c_2\right)\right)\right)^2 + c_1 - c_2\\
    &\quad -\left(\frac{1}{2\,\sqrt{c_1}}\left(2 \,c_1 - c_2\right)\right)^2\!\!.
\end{align}
One can notice that
\begin{equation}
     c_1 - c_2-\left(\frac{1}{2\,\sqrt{c_1}}\left(2 \,c_1 - c_2\right)\right)^2 = -\frac{c_2^2}{4\,c_1}.
\end{equation}
Thus, we have
\begin{align}
    e^{k\,\tau_i} &= \left( \sqrt{\frac{\xi}{c_1}  +\frac{c_2^2}{4\,c_1^2}}+\frac{1}{2\,c_1}\left(2 \,c_1 - c_2\right)\right)\!.
\end{align}
This implies
\begin{align}
    \tau_i &= \frac{1}{k} \log\left( \sqrt{\frac{\xi}{c_1}  +\frac{c_2^2}{4\,c_1^2}}+1 - \frac{c_2}{2\,c_1}\right).
\end{align}
Since $\sqrt{\frac{\xi}{c_1}  +\frac{c_2^2}{4\,c_1^2}} - \frac{c_2}{2\,c_1} >0$, then $\tau_i > 0$.
\end{proof}
\end{document}